\documentclass{article}
\usepackage[utf8]{inputenc}
\usepackage{amsmath}
\usepackage{amsthm}
\usepackage{amsfonts}
\usepackage{amssymb}
\usepackage{tabu}
\usepackage[margin=1in]{geometry}
\usepackage{latexsym} 
\usepackage{enumerate} 
\usepackage[colorlinks=true,pagebackref]{hyperref}
\usepackage{mathrsfs}
\usepackage{titlesec} 

\titleformat{\section}[block]{\large\scshape\centering}{\thesection.}{1em}{} 
\titleformat{\subsection}[block]{\scshape\centering}{\thesubsection.}{1em}{} 

\newtheorem{theorem}{Theorem}[section]
\newtheorem*{theorem*}{Theorem}
\newtheorem{lemma}[theorem]{Lemma}
\newtheorem*{lemma*}{Lemma}
\newtheorem{proposition}[theorem]{Proposition}
\newtheorem{corollary}[theorem]{Corollary}
\newtheorem*{corollary*}{Corollary}
\newtheorem{definition}[theorem]{Definition}

\theoremstyle{definition}
\newtheorem{example}[theorem]{Example}
\newtheorem{remark}[theorem]{Remark}
\def\Ozgur{{\"O}zg{\"u}r}

\def\CC{\mathbb C}

\def\RR{\mathbb R}

\def\ZZ{\mathbb Z}

\def\ca{\mathrm{ca}}
\def\co{\mathcal{C}}

\newcommand\module{\mathrm{mod}}
\newcommand\MCM{\mathrm{MCM}}

\DeclareMathOperator{\coker}{coker}
\DeclareMathOperator{\sEnd}{\underline{End}}
\DeclareMathOperator{\st}{st}
\DeclareMathOperator{\sHom}{\underline{Hom}}
\newcommand\sMCM{\underline{\mathrm{MCM}}}
\DeclareMathOperator{\gldim}{gldim}

\DeclareMathOperator{\ann}{ann}

\DeclareMathOperator{\annT}{ann_{\mathcal{T}}}

\DeclareMathOperator{\Hom}{Hom}

\DeclareMathOperator{\HomT}{Hom_{\mathcal{T}}}

\DeclareMathOperator{\End}{End}

\DeclareMathOperator{\EndT}{End_{\mathcal{T}}}

\DeclareMathOperator{\Ext}{Ext}
\DeclareMathOperator{\sExt}{\underline{Ext}}

\DeclareMathOperator{\sann}{\underline{ann}}

\def\TT{\mathcal{T}}

\def\nrml{\overline}

\usepackage{tikz}
\usetikzlibrary{matrix,arrows,decorations.pathmorphing}
\usepackage{tikz-cd}
\usepackage[all]{xy}

\title{\scshape{The Cohomology Annihilator of a Curve Singularity}}
\author{\scshape{\Ozgur \; Esentepe} \footnote{Department of Mathematics, University of Toronto, Toronto, ON M5S 2E4 Canada; e-mail: ozgures@math.toronto.edu}}
\date{}

\begin{document}

\maketitle

\begin{center}
    \textit{Dedicated to the memory of Ragnar-Olaf Buchweitz}
\end{center}

\vspace{1cm}

\begin{abstract}
The aim of this paper is to study the theory of cohomology annihilators over commutative Gorenstein rings. We adopt a triangulated category point of view and study the annihilation of stable category of maximal Cohen-Macaulay modules. We prove that in dimension one the cohomology annihilator ideal and the conductor ideal coincide under mild assumptions. We present a condition on a ring homomorphism between Gorenstein rings which allows us to carry the cohomology annihilator of the domain to that of the codomain. As an application, we generalize the Milnor-Jung formula for algebraic curves to their double branched covers. We also show that the cohomology annihilator of a Gorenstein local ring is contained in the cohomology annihilator of its Henselization and in dimension one the cohomology annihilator of its completion. Finally, we investigate a relation between the cohomology annihilator of a Gorenstein ring and stable annihilators of its noncommutative resolutions.
\end{abstract}

\section{Introduction}

The theory of cohomology annihilators have been investigated in a series of papers by Iyengar and Takahashi in the last decade. The authors have utilized cohomology annihilators in order to study generation of categories (abelian or triangulated). The aim of this paper is to study the theory of cohomology annihilators over Gorenstein rings. In this setting, the cohomology annihilator ideal is the annihilator of the stable category of maximal Cohen-Macaulay modules.

We start by collecting some known facts about annihilation of a category which is linear over a commutative ring. This is the content of Section 2. We say that a ring element annihilates such a category if and only if it annihilates the endomorphism ring of every indecomposable object. We show that in the case of commutative Gorenstein rings, the annihilator of the stable category of maximal Cohen-Macaulay modules is the cohomology annihilator ideal. 
 \begin{lemma*}[\ref{mainlemma}]
 	For a Gorenstein ring $R$, the cohomology annihilator ideal is the annihilator of $\sMCM(R)$.
 \end{lemma*}
We also describe the annihilation in terms of complete resolutions and matrix factorizations.

We devote Section 3 to examples. These examples and their relations to Jacobian ideals are the main motivation for this project. Indeed, it is known that for a reasonably large class of rings the cohomology annihilator ideal contains the Jacobian ideal. It is an observation due to Ragnar-Olaf Buchweitz that for one dimensional local rings of finite maximal Cohen-Macaulay type if the codimension of the Jacobian ideal is even, then the codimension of the cohomology annihilator ideal is exactly the half of the former. This observation, which was privately shared with the author of this paper, is the starting point of this project. We use the examples of this section in Section 6 when we talk about the stable annihilator of a noncommutative resolutions.
	
In Section 4, we show that the conductor ideal of a one dimensional commutative Gorenstein ring contains the cohomology annihilator ideal. Due to a result of Wang, the conductor of a one dimensional Noetherian reduced complete local ring is contained in the cohomology annihilator ideal. Combining Wang's result with ours, we identify the cohomology annihilator ideal with the conductor ideal. 

\begin{theorem*}[\ref{maintheorem}]
 Let $ R $ be a one dimensional reduced complete Gorenstein local ring. Then,
		\begin{align*}
		\ca(R) = \co(R).
		\end{align*}
\end{theorem*}

The result of Section 4 explains Buchweitz's observation. The main ingredient here is the Milnor-Jung formula from geometry. Replacing the delta invariant with the codimension of the cohomology annihilator, we generalize the Milnor-Jung formula to double branched covers of plane curves. We also investigate the relation between the cohomology annihilators of a hypersurface and its $n$-branched covers (See Section 5).
	
Section 5 is devoted to the weak MCM extending property of a ring homomorphism between two Gorenstein rings. We show that if this property is satisfied, then the image of the cohomology annihilator of the domain is contained in the cohomology annihilator of the codomain. We give three applications of this one of which is the generalization of Milnor-Jung formula as explained above. 

 \begin{theorem*}[\ref{mainbranchedtheorem}]
    Suppose $k$ is an algebraically closed field with $\mathrm{char} k \neq 2$. Let $R = k[\![ x,y]\!] / (f)$ be the coordinate ring of a reduced curve singularity. For $S = k[\![x,y, z_1, \ldots, z_l]\!]/(f+z_1^2 + \ldots + z_l^2)$, we have
    \begin{align*}
     \dim_k \frac{S}{J_{S}} = 2 \dim_k \frac{S}{\ca(S)} -  r + 1
    \end{align*}
    where $r$ is the number of branches of the curve $f$ at its singular point.
   \end{theorem*}
We also show that in dimension one, the cohomology annihilator of a Gorenstein local ring is equal to the intersection of the ring and the cohomology annihilator of its completion (provided that the completion is reduced). In another direction, with no assumption on the dimension, we show that the cohomology annihilator of a Gorenstein local ring is contained in the cohomology annihilator of its Henselization.

Most of our examples have a single object whose stable annihilator is equal to the cohomology annihilator ideal. Moreover, in all of these examples this single object can be chosen as a noncommutative resolution. In Section 6, we present examples of this. We study the relation between the cohomology annihilator of a Gorenstein ring and the stable annihilator of a noncommutative resolution. We show that the two are equal up to radicals in general.

\begin{theorem*}[\ref{ncrtheorem}]
 Let $R$ be a Gorenstein ring and $\Lambda$ be a finite $R$-algebra such that the canonical map $R \to \Lambda$ is a split monomorphism. Assume that $\gldim \Lambda = \delta < \infty$ and that $\Lambda \in \MCM(R)$. Then,
 \begin{align*}
  (\sann_R\Lambda)^{\delta + 1} \subseteq \ca(R) \subseteq \sann_R(\Lambda).
 \end{align*}
\end{theorem*}
 	
\paragraph{Acknowledgements.} My biggest thanks are due to Ragnar-Olaf Buchweitz, my supervisor and mentor, who introduced me to the motivating problem of this paper. Thanks are also due to Graham Leuschke, my cosupervisor, who encouraged me in various stages of this project. I also want to thank Benjamin Briggs and Vincent Gelinas for fruitful discussions at University of Toronto. Finally, a thank you to Srikanth Iyengar and Ryo Takahashi for their hospitality and their comments on a preliminary version of this paper.

\section{Background and Preliminaries}

In this section, we present definitions, preliminary lemmas and notation. Throughout this paper, $R$ will always denote a commutative Noetherian ring and all modules are assumed to be finitely generated.
 
\subsection{Annihilation of a Triangulated Category.} Let $R$ be a commutative ring and $\TT$ be an $R$-linear category with shift functor $\Sigma$. We say that $r \in R$ annihilates $X \in \TT$ if $x \in \ann_R \EndT X$ and we write $r \in \annT X$ in this case. We say that $r$ annihilates $\TT$ if $r \in \ann_R \EndT X$ for every $X \in \TT$. We put
\begin{align*}
 \ann_R \TT = \{ r \in R \; : \; r \text{ annihilates } \TT\}
\end{align*}

Note that for any $X, Y \in \TT$, $\HomT(X,Y)$ is naturally a module over $\EndT X$ and therefore $\annT X \subseteq \ann_R \HomT(X,Y)$. Consequently, we have $\annT(X \oplus Y) = \annT X \cap \annT Y$. On the other hand, $\EndT X \cong \EndT \Sigma X$ and so $\annT X  = \annT \Sigma X$. Hence, we conclude that if $r \in \annT X$, then $r$ annihilates every object in $\langle X \rangle$ - the smallest subcategory containing $X$ and closed under finite direct sums, direct summands and $\Sigma$.
 
We put $\langle X \rangle_1 = \langle X \rangle$. Inductively, we define $\langle X \rangle_n$ to be the full subcategory of $\TT$ consisting of objects $B$ such that there is a distinguished triangle $A \to B \to C \to \Sigma A$ with $A \in \langle X \rangle_{n-1}$ and $C \in \langle X \rangle$. We say that $X$ is a strong generator for $\TT$ if $\TT = \langle X \rangle_n$ for some integer $n$. Note that applying the cohomological functor $\HomT(-,B)$ to the triangle $A \to B \to C \to \Sigma A$ yields an exact sequence
 \begin{align*}
   \HomT(C,B) \to \EndT(B) \to \HomT(A,B).
  \end{align*}
  Therefore, $\ann_R \HomT(C,B) \ann_R \HomT(A,B) \subseteq \ann_R \EndT(B) $. The following lemma summarizes this discussion.
  \begin{lemma}\label{triangulatedlemma}
  Let $\TT$ be an $R$-linear triangulated category with shift functor $\Sigma$ and $X \in \TT$. For any $n \geq 1$ and $Y \in \langle X \rangle_n$, we have an inclusion $(\annT X )^n \subseteq \annT Y$. In particular, if $X$ is a strong generator with $\TT = \langle X \rangle_n$ then $(\annT X )^n$ annihilates $\TT$.
  \end{lemma}
  
\subsection{Stable Category of Maximal Cohen-Macaulay Modules over Gorenstein Rings.} The triangulated category we are interested in this paper is the stable category of maximal Cohen-Macaulay modules. We recall the theory of maximal Cohen-Macaulay modules over Gorenstein rings. We refer to \cite{B} for a detailed study.
  
A Noetherian ring $R$ is called (strongly) Gorenstein if it has finite injective dimension both as a left module and as a right module over itself. Now assume that $R$ is Gorenstein. An $R$-module $M$ is called maximal Cohen-Macaulay if $\Ext_R^i(M, R) = 0$ for $ i > 0$. We denote the full subcategory of maximal Cohen-Macaulay modules by $\MCM (R)$. Let $P(M,N)$ denote the submodule of all $R$-linear homomorphisms from $M$ to $N$ which factor through a projective $R$-module. The stable category of maximal Cohen-Macaulay modules, denoted $\sMCM(R)$, is the category whose objects are the same as $\MCM(R)$ and whose morphism sets are 
  \begin{align*}
      \sHom_R( - , * ) := \Hom_R( - , *) / P(-, *).
  \end{align*}
  On $\sMCM(R)$, the syzygy functor $\Omega$ has an inverse, namely the cosyzygy functor $\Omega^{-1}$. The cosyzygy functor gives $\sMCM(R)$ the structure of an $R$-linear triangulated category and we have an $R$-linear triangulated equivalence
  \begin{align*}
      \sMCM(R) \cong D^b_{sg}(R) := D^b(R) /\mathrm{perf} R
  \end{align*}
  where $D^b_{sg}(R)$ is the stabilized derived category or the singularity category obtained by taking the Verdier quotient of the bounded derived category by the subcategory of perfect complexes. For any $R$-module $M$, seen as an object in $D^b_{sg}(R)$, the maximal Cohen-Macaulay module corresponding to $M$ under this equivalence is called the maximal Cohen-Macaulay approximation (or the stabilization) of $M$ and is denoted by $M^{\st}$.
  
Let $M$ and $N$ be two (complexes of) finitely generated $R$-modules. The $n$th Tate cohomology group is defined to be 
  \begin{align*}
      \sExt_R^n(M,N):= \mathrm{Hom}_{D^b_{sg}(R)}(M, N[n]) = \sHom_R(\Omega^n M^{\st}, N^{\st})
  \end{align*}
 for any $n \in \ZZ$. Note that these groups are possibly nonzero for negative integers as well. They are sometimes called stable Ext groups because because $\sExt_R^n(M,N) = \Ext_R^n(M,N)$ if $n$ is larger than the Krull dimension of $R$. If $M, N$ are maximal Cohen-Macaulay modules, we have $\sExt_R^0(M,N) = \sHom_R(M,N)$ and also $\sExt_R^{>0}(M,N) = \Ext_R^{>0}(M,N)$
 
 \subsection{Cohomology Annihilators.} Following Iyengar and Takahashi \cite{IT}, we define the \textit{cohomology annihilator} ideal as follows:
		\begin{definition}
			The $ n $th cohomology annihilator, $ \ca^n(R) $, of a commutative Noetherian ring $ R $ is
			\begin{align*}
			\ca^n(R) = \bigcap_{M, N \in \module R} \ann_R\Ext_R^{n}(M, N)
			\end{align*}
			The cohomology annihilator, $ \ca(R) $, of $ R $ is the union
			\begin{align*}
			\ca(R) = \bigcup_{n \geq 0}\ca^n(R).
			\end{align*}
		\end{definition}
		Recall that for any two finitely generated $R$-modules $M$ and $N$, and any $n>0$ one has 
		\begin{align*}
		 \Ext_R^{n+1}(M,N) \cong \Ext_R^{n}(\Omega M,N)
		\end{align*}
		where $\Omega M$ is a syzygy of $M$. Hence, 
		\begin{align*}
		 \ann_R\Ext_R^{n+1}(M,N) = \ann_R\Ext_R^{n}(\Omega M,N)
		\end{align*}
		which implies that $\ca^{n}(R) \subseteq \ca^{n+1}(R)$. So, we have an ascending chain $\ca^0(R) \subseteq \ca^1(R) \subseteq \ca^2(R) \subseteq \ldots$. As $R$ is Noetherian, this chain stops. So, $\ca(R) = \ca^s(R)$ for some $s \gg 0$.

	 \begin{lemma}\label{mainlemma}
 	For a Gorenstein ring $R$, the cohomology annihilator ideal is the annihilator of $\sMCM(R)$.
 \end{lemma}
\begin{proof}Recall that $ \ca(R) = \ca^s(R) $ for some large $ s $. Picking $s$ larger than the Krull dimension, we have
\begin{align*}
\ca(R) &= \bigcap_{M, N \in \module R} \ann_R\Ext_R^{s}(M, N) = \bigcap_{M, N \in \module R} \ann_R\sExt_R^{s}(M, N)
\\&= \bigcap_{M, N \in \module R} \ann_R\sExt_R^{s}(M^{\st}, N^{\st})= \bigcap_{M, N \in \MCM R} \ann_R\sExt_R^{s}(M, N)
\\&= \bigcap_{M, N \in \MCM R} \ann_R\sExt_R^{0}(\Omega^s M, N)
\\&= \bigcap_{M, N \in \MCM R} \ann_R\sHom_R(\Omega^s M, N)
\end{align*}
If one of $ \Omega^s M $ and $ N $ is projective, then the $ \sHom_R $ group is zero and therefore the annihilator of that $ \sHom_R $ group is the entire ring. So, we can focus on the non-projective maximal Cohen-Macaulay modules. Moreover, since $ \Omega $ is an autoequivalence of $ \sMCM(R) $, we can write the cohomology annihilator as
\begin{align*}
\ca(R) = \bigcap_{M, N \in \sMCM R} \ann_R\sHom_R(M, N)
\end{align*}
If $r$ annihilates the stable endomorphism ring $\sEnd_R(M)$ for a maximal Cohen-Macaulay module $M$, then it annihilates $\sHom_R(M,N)$ for any $N \in \MCM (R)$. This finishes the proof.
\end{proof}
Recall that for any commutative ring $A$ and any $A$-module $X$, one has $\mathrm{ann}_A(X) = \mathrm{ann}_A \mathrm{End}_A(X)$. Hence, for ease of notation, we make the following definition.

\begin{definition}
 Let $R$ be a commutative Gorenstein ring. For any $M \in \sMCM(R)$, the stable annihilator of $M$ is
 \begin{align*}
  \sann_R(M) := \ann_R \sEnd_R(M) = \ann_R \mathrm{End}_{\sMCM(R)}(M).
 \end{align*}
\end{definition}

\begin{remark}{\label{dualremark}}
 Let $M$ be a maximal Cohen-Macaulay module. Let $P(M)$ and $\CC\RR(M)$ be a projective resolution and a complete resolution of $M$, respectively. It is sometimes useful to compute the stable annihilator of $M$ in terms of $P(M)$ and $\CC\RR(M)$. We note that an element $r \in \sann_R(M)$ if and only if the multiplication by $r$ 
\begin{enumerate}
 \item on $M$ factors through a projective module, or
 \item on $P(M)$ factors through a perfect complex, or
 \item on $\CC\RR(M)$ is nullhomotopic.
\end{enumerate}
For instance, the equality $\sann_R(M) = \sann_R(\Omega M)$ amounts to observing that if $\CC\RR(M)[1] = \CC\RR(\Omega M)$ and multiplication by $r$ is nullhomotopic on $\CC\RR(M)$ if and only if it is nullhomotopic on $\CC\RR(M)[1]$. One can also easily see that $\sann_R(M) = \sann_R(M^*)$ where $M^* = \Hom_R(M,R)$, by dualizing the commutative diagram \begin{align*}
	\xymatrix{
		M \ar[dr]_{\alpha} \ar[rr]^{r} && M\\
		& F \ar[ur]_{\beta}
	}
	\end{align*}
with $F$ being a free module. 
\end{remark}

\section{Examples}
In this section, we present examples of cohomology annihilator computations.
 
 \begin{remark}\label{twoperiodicityremark}
 If $R$ is a hypersurface singularity and $M$ is a maximal Cohen-Macaulay module, then any minimal complete resolution of $M$ is 2-periodic. Therefore, $\sEnd_R(M) \cong \sExt_R^2(M,M) \cong \Ext_R^2(M,M)$ \cite{Eis}. This means that the stable annihilator of $M$ is equal to $\ann_R\Ext_R^2(M,M)$.
 
 More importantly, the 2-periodicity tells us that we can compute the stable annihilators using matrix factorizations. In particular, let $R = k[\![x_1, \ldots, x_m]\!]/(f)$ and $A$ and $B$ be two $n \times n$ matrices such that $AB = BA = 0$ and $\coker(A) = M$. Then, $r$ stably annihilates $M$ if and only if there are matrices $C$ and $D$ such that the following diagram commutes. 
 
 \begin{align*}
 	\xymatrix{
 		R^n \ar[d]^r \ar[r]^{A} & R^n \ar[d]^r  \ar[r]^{B} \ar@{..>}[dl]^C & R^n \ar[d]^r  \ar@{..>}[dl]^D \\
 		R^n \ar[r]^{A} & R^n \ar[r]^{B} & R^n 
 	}
 	\end{align*}
 This description is not only useful in computations but also it allows to prove the Jacobian ideal is contained in the cohomology annihilator. By using this description of annihilation, it follows from the product rule for derivatives that $\partial f / \partial x_i$ is contained in the cohomology annihilator for all $i$.
 
\end{remark}

\begin{example}\label{onedimensionalexample}
	Let $ R = k[\![ x,y ]\!]/(f) $ be a one dimensional local ring of finite representation type where $k$ is an algebraically closed field of characteristic zero. Then, a complete list of indecomposable maximal Cohen-Macaulay modules is given in \cite[Ch. 9]{Y} in terms of matrix factorizations. We can compute the cohomology annihilator in Singular. We illustrate this on the $A_n$ case. The following Singular code computes $\sann_R \Omega \coker A$ where 
 	\begin{align*}
 	A = A_j = \begin{bmatrix}
 	x & y^j \\ y^{n+1-j} & -x
 	\end{bmatrix}.
 	\end{align*}
 
 	\begin{verbatim}
 		LIB ``homolog.lib'';
 		LIB ``primdec.lib'';
 		ring S = 0, (x,y), ds;
 		int n = n ; \\write your n here
 		ideal I = x2 + y^(n+1);
 		qring R = std(I);
 		int j = j ; \\write your j here
 		int k = n+1 - j;
 		matrix A[2][2] = x, yj , yk, -x;
 		print(std(Ann(Ext(2,A,A))));
 	\end{verbatim}
 	
Cokernels of these matrices for $j = 0, \ldots, n+1$ give a complete list of indecomposable maximal Cohen-Macaulay modules. Computing their stable annihilators with this code, we see that the cohomology annihilator for the $ A_{n} $ case is the ideal $ (x, y^{\left\lfloor n \rfloor\right/2}) $.
 \end{example}

 	\begin{example}\label{twodimensionalexample}
 		Let $G$ be a finite subgroup of $SL(2,k)$ where $k$ is an algebraically closed field of characteristic zero. The group $G$ acts on $S = k[\![x,y]\!]$ linearly and one can consider the invariant ring
 		\begin{align*}
 		R = S^G = \{ r \in S : g(r) = r \text{ for all } g \in G \}.
 		\end{align*}
 		Then, $R$ is Gorenstein and in fact isomorphic to a two dimensional ADE singularity. In particular, $\dim S = \dim R = 2$ and thus $S \in \MCM(R)$. Moreover, $\mathrm{add} S = \MCM(R)$ see \cite[Proposition 10.5]{Y}. Hence, by Lemma \ref{triangulatedlemma} we deduce that 
 		\begin{align*}
 		\ca(R) = \sann_R (S).
 		\end{align*}
 	\end{example}
 	
 	\begin{example}\label{determinantalex}
 	    Let $k$ be a field and $X = (x_{ij})$ be the generic $n \times n$ square matrix. Consider the $n^2 - 1$ dimensional hypersurface ring $R = k[X]/\det(X)$. Let $M = \coker X$. Then, $M \in \MCM(R)$. We have an isomorphism $\sEnd_R(M) \cong R/ J_R$ \cite[Corollary 5.6.]{BL}. Hence, we conclude that $\sann_R(M) = J_R$. Thus, $\ca(R) \subseteq J_R$. We already know that the Jacobian ideal is contained in $\ca(R)$ so we have the equality $\ca(R) = J_R$. Here, $J_R$ is equal to the ideal generated by the submaximal minors of $X$ \cite[Lemma 2.6.]{BL}.
 	\end{example}
 	
 	\begin{example}\label{torusinvariant}
 	    In this example, we compute the cohomology annihilator of the one dimensional torus invariant ring with weights $[2,1,-2,-1]$. See \cite[Chapter 16]{Y} and \cite[Section 8]{vdb} for reference. Let $k$ be an algebraically closed field of characteristic zero. Suppose that $T = k^*$ acts on a four dimensional vector space spanned by $x_1, x_2, x_3, x_4$ with weights $a_1 = 2, a_2 = 1, a_3 = -2, a_4 = -1$. That is, for any $t \in T$, $t \cdot x_i = t^{a_i} x_i$ and this action extends linearly. This action extends to an action on the polynomial ring $S = k[x_1, \ldots, x_n]$. Let $R = S^T$ be the invariant ring. As the weights add up to zero, $R$ is Gorenstein.
 	    
Assigning $\deg x_i = a_i$, we can make $S$ into a $\ZZ$-graded ring. Let $S_c$ be the degree $c$ part. Clearly, $S_0 = R$ and every $S_c$ has an $R$-module structure. Moreover, $S_1$ and $S_2$ are maximal Cohen-Macaulay modules over $R$. Hence, $\ca(R) \subseteq \sann_R S_1 \cap \sann_R S_2$.
 	  
Note that for any $-2 \leq a,b \leq 2$, we have 
 	    \begin{align*}
 	        \Hom_R(S_a, S_b) \cong S_{b-a}.
 	    \end{align*}
 	    In particular, $S_a^* = \Hom_R(S_a, S_0) \cong S_{-a}$ and $\End_R(S_a) = R$. Hence, via the exact sequence
 	    \begin{align*}
 	        S_a^* \otimes_R S_a \to \End_R(S_a) \to \sEnd_R(S_a) \to 0 
 	    \end{align*}
 	    one can see the stable annihilator of $S_a$ is equal to the image of the multiplication map $S_{-a} \otimes_R S_a \to R$. It is easy to show that 
 	    \begin{align*}
 	        S_0 & = R = k[x_1x_3, x_1x_4^2, x_2^2x_3, x_3x_4]\\
 	        S_{-1} & = R x_4 + R x_2x_3\\
 	        S_1 & = Rx_2 + Rx_1x_4 \\
 	        S_{-2} & = Rx_3 + Rx_4^2 \\
 	        S_{2} & = Rx_1 + Rx_2^2
 	    \end{align*}
 	    and so
 	    \begin{align*}
 	        \sann_R S_1 &= S_{-1}S_1 = (x_2x_4, x_2^2x_3, x_1x_4^2, x_1x_2x_3x_4) \\
 	        \sann_R S_2 & = S_{-2}S_2 = (x_1x_3, x_1x_4^2, x_2^2x_3, x_2^2x_4^2).
 	    \end{align*}
 	    Next, notice that we have a $k$-algebra isomorphism
 	    \begin{align*}
 	        R = k[x_1x_3, x_1x_4^2, x_2^2x_3, x_3x_4] \cong \frac{k[x,y,z,w]}{(xw^2 - yz)} = P.
 	    \end{align*}
 	    Under this isomorphism, we have the identifications
 	    \begin{align*}
 	        \sann_R S_1 & \leftrightarrow (w,z,y) \\
 	        \sann_R S_2 & \leftrightarrow (x,y,z,w^2)
 	    \end{align*}
 	    and hence
 	    \begin{align*}
 	        \sann_R S_1 \cap \sann_R S_2 \leftrightarrow (w,z,y) \cap (x,y,z,w^2) = (xw, y, z, w^2) =J_P
 	    \end{align*}
 	    where $J_P$ denotes the Jacobian ideal of $P$ generated by the partial derivatives of $xw^2 -yz$. We conclude that $\ca(R)$ is contained in the Jacobian ideal $J_R$ of $R$. We also know that the Jacobian ideal is contained in the cohomology annihilator. So, we conclude that
 	     \begin{align*}
 	         \ca(R) = \sann_R S_1 \cap \sann_R S_2 = J_R.
 	     \end{align*}
 	\end{example}

\section{One Dimensional Gorenstein Rings}
 In this section, we describe the cohomology annihilator of one dimensional Gorenstein rings. The main result is that under reasonable assumptions, the cohomology annihilator is equal to the conductor ideal.
 
 \subsection{Conductor Ideals}
 
Let $R$ be a commutative Noetherian ring. The \textit{normalization} of $R$, denoted by $\nrml{R}$, is the integral closure of $R$ inside its total quotient ring. The \textit{conductor} ideal of $R$, denoted by $\co(R)$, is the ideal 
 \begin{align*}
		\{r \in R \; : \; rq \in R \text{ for all } q \in \nrml{R}\} = \ann_R\left( \nrml{R}/R \right).
 \end{align*}
 We will use the following well-known facts about the conductor ideal. References for these results include \cite{SH, GLS, O}.
 
 \begin{enumerate}
     \item The normalization $\nrml{R}$ is (module-)finite over $R$ if and only if $\co(R)$ contains a nonzerodivisor.
     \item In this case, we have
     \begin{align*}
         \End_R(\nrml{R}) \cong \nrml{R}
     \end{align*}
     and
     \begin{align*}
         \Hom_R(\nrml{R}, R) \cong \co(R) \quad .
     \end{align*}
     \item A local ring has finite normalization if its completion is reduced. If the ring is Cohen-Macaulay of dimension 1, then the converse is also true \cite[Proposition 4.6.]{LW}
     \item The conductor is also an ideal of $\nrml{R}$. In fact, it is the largest common ideal of $R$ and $\nrml{R}$.
 \end{enumerate}
 
 \begin{proposition}
		Suppose that $R$ has finite normalization. Then, as $ R $-modules we have
		\begin{align*}
		\sEnd_R(\nrml{R}) \cong \frac{\nrml{R}}{\co(R)}.
		\end{align*}
 \end{proposition}
 
 \begin{proof}
     For any two finitely generated $R$-modules $M$ and $N$, we have the following exact sequence
     \begin{align*}
         \Hom_R(M,R) \otimes_R N \xrightarrow{\nu} \Hom_R(M,N) \to \sHom_R(M,N) \to 0
     \end{align*}
     where $\nu(\phi \otimes n)$ with $\phi \in \Hom_R(M,R)$ and $n \in N$ is the map which takes $m \in M$ to $\phi(m)n$. Picking $M = N = \nrml{R}$, we then get
     \begin{align*}
         \Hom_R(\nrml{R},R) \otimes_R \nrml{R} \xrightarrow{\nu} \End_R(\nrml{R}) \to \sEnd_R(\nrml{R}) \to 0
     \end{align*}
     which gives
     \begin{align*}
         \co(R) \otimes \nrml{R} \xrightarrow{\mu} \nrml{R} \to \sEnd_R(\nrml{R}) \to 0
     \end{align*}
     where $\mu$ is the multiplication map. As the conductor is an ideal of the normalization, we get that the image of $\mu$ lies in the conductor. Moreover, $\co(R) = \mu(\co(R) \otimes_R 1) $ is in the image of $\mu$. Hence, we have that the image of $\mu$ is equal to the conductor which finishes the proof.
 \end{proof}
 
	\begin{corollary}\label{firstcorollary}
		Suppose that $R$ has finite normalization. Then, we have
		\begin{align*}
		\ann_R(\sEnd_R(\nrml{R})) = \co(R).
		\end{align*}
	\end{corollary}

	\begin{proof}We have
		\begin{align*}
		\ann_R(\sEnd_R(\nrml{R})) = \ann_R\left( \frac{\nrml{R}}{\co(R)} \right) = \ann_R\left( \nrml{R} / R\right) = \co(R).
		\end{align*}
	\end{proof}
 Suppose that $ R $ is a one dimensional Gorenstein local ring with finite normalization. Then, $ \nrml{R} $ is a one dimensional regular ring. Moreover, we have $ \nrml{R} $ is maximal Cohen-Macaulay as an $ R $-module as it is torsion-free over $R$. Therefore, Lemma \ref{mainlemma} gives
	\begin{align*}
	\ca(R) \subseteq \ann_R\sEnd_R(\nrml{R}) = \co(R).
	\end{align*}
	
Let $ R $ be a one dimensional reduced complete Noetherian local ring. Then, $ \co(R) $ annihilates $ \Ext_R^1(M,N) $ for any $ M \in \MCM(R) $ and $ N \in \module(R) $\cite[Corollary 3.2]{W}. In particular, we have that 
	\begin{align*}
	\co(R) \subseteq \ann_R \Ext_R^1(\Omega^{-1}M, M) = \ann_R \sExt_R^1(\Omega^{-1}M, M) = \ann_R \sEnd_R(M)
	\end{align*}
	for all $ M\in \MCM(R) $. For us, this means $ \co(R) \subseteq \ca(R) $. Therefore, we have the following result.
	\begin{theorem}\label{maintheorem}
		Let $ R $ be a one dimensional reduced complete Gorenstein local ring. Then,
		\begin{align*}
		\ca(R) = \co(R).
		\end{align*}
	\end{theorem}

  \subsection{Buchweitz's Observation Explained}As stated in the introduction, the starting point of this project was an observation due to Ragnar-Olaf Buchweitz on the relation between codimensions of the Jacobian ideal and the cohomology annihilator ideal.
  
  Let $R$ be a equicharacteristic complete Noetherian local ring with coefficient field $k$. Then, by Cohen's Structure Theorem $R$ is isomorphic to a ring of the form $k[\![ x_1, \ldots, x_n]\!]/(f_1, \ldots, f_m)$. The ideal generated by all maximal minors of the Jacobian matrix - whose entries are given by the partial derivatives $\partial f_j/ \partial x_i$ - is called the Jacobian ideal $J_R$ of $R$. If we assume that $f_1, \ldots, f_m$ form a regular sequence so that $R$ is a complete intersection ring, then $J_R$ is contained in the cohomology annihilator $\ca(R)$ \cite[Corollary 7.8.7.]{B}. For the case $m = 1$, that is the hypersurface case, this follows from Remark \ref{twoperiodicityremark}. In \cite{ITJac}, authors define the Jacobian ideal for a larger family of rings and they prove that the Jacobian ideal is still contained in the cohomology annihilator in this case.
  
   From now on, let us assume that $R = k[\![x,y]\!]/(f)$ is the coordinate ring of a hypersurface singularity. We have the containments
  \begin{align*}
   J_R \subseteq \ca(R) \subseteq R.
  \end{align*}
  Buchweitz observed that if $f$ is one of the polynomials $x^2 + y^{n+1}, x^3 + y^4$ or $ x^3+ y^5$, then one has
  \begin{align*}
   \dim_k \frac{R}{J_R} = 2 \dim_k \frac{R}{\ca(R)}.
  \end{align*}
  We can add more examples, using Theorem \ref{maintheorem}.
  
  \begin{example}
  Let $f = x^a + y^b$ where $a,b > 1$ are coprime so that the semigroup generated by $a$ and $b$ contains all but finitely many nonnegative integers. In particular, the Frobenius number - the last number which does not appear in the semigroup - is given by $(a-1)(b-1) - 1$. Moreover, since $R$ is Gorenstein, so is $T = k[\![t^a, t^b]\!]$, which implies the semigroup generated by $a$ and $b$ is symmetric. That is, if two numbers add up to $(a-1)(b-1) - 1$, then exactly one of them is in the semigroup. Therefore, exactly half of the numbers between $0$ and $(a-1)(b-1) - 1$ are in the semigroup. Since the conductor ideal $\co(T)$ of $T$ contains every power of $t$ larger than $(a-1)(b-1) - 1$, one can conclude
  \begin{align*}
    \dim_k \frac{R}{\ca(R)} = \dim_k \frac{T}{\co(T)} = \frac{(a-1)(b-1)}{2}.
  \end{align*}
  On the other hand, the Jacobian ideal $J_R$ is generated by the two elements $x^{a-1}$ and $y^{b-1}$ assuming that the characteristic of $k$ divides neither $a$ nor $b$. Thus,
  \begin{align*}
   \dim_k \frac{R}{J_R} = \dim_k \frac{k[\![x,y]\!]}{(f, x^{a-1}, y^{b-1})} = \dim_k \left( \frac{ k[\![x]\!] }{ (x^{a-1})} \otimes_k \frac{ k[\![y]\!] }{ (y^{b-1})} \right) = (a-1)(b-1)
  \end{align*}
  
  Combining the two equalities, we extend Buchweitz's observation to rings of the form $R= k[\![x,y]\!]/(x^a + y^b)$ where $a$ and $b$ are coprime.
  \end{example}
  \begin{remark} We can extend this observation even further. Suppose that $R = k[\![x,y]\!]/(f)$ is the coordinate ring of a reduced curve singularity. Then, the common number
  \begin{align*}
   \dim_k \frac{\nrml{R}}{R} = \dim_k \frac{R}{\co(R)}
  \end{align*}
  is called the \textit{delta invariant} of the curve - which we denote by $\delta_R$. By Theorem \ref{maintheorem}, we have
  \begin{align*}
   \delta_R = \dim_k \frac{R}{\ca(R)}.
  \end{align*}
  On the other hand, the number
  \begin{align*}
   \mu_R = \dim_k \frac{R}{J_R}
  \end{align*}
  is called the \textit{Milnor number} of $R$. By letting $r$ be the number of branches of the curve at the unique singular point, namely the origin, we have the Milnor-Jung Formula
  \begin{align*}
   \mu_R = 2\delta_R - r + 1.
  \end{align*}
  This formula returns Buchweitz's observation
  \begin{align*}
   \dim_k \frac{R}{J_R} = \mu = 2 \delta =  2 \dim_k \frac{R}{\ca(R)}
  \end{align*}
  for unibranch $(r = 1)$ singularities. Indeed, the singularities from Example \ref{onedimensionalexample} are all unibranch. See \cite[Section 4.2.4.]{Dol} and \cite{Hun} for references.
 \end{remark}

  \section{Cohomology Annihilators and Weak MCM Extending Property}
  Let $f: R \to S$ be a ring homomorphism of Gorenstein rings. We say that \textit{$f$ has the weak MCM extending property} if for any $N \in \MCM(S)$, there is an $M \in \MCM(R)$ and $N' \in \MCM(S)$ such that $N'$ is a direct summand of $M \otimes_R S$ and $N \in \langle N' \rangle$. The purpose of this section is to study the behaviour of cohomology annihilator ideal under maps with the weak MCM extending property.
  
  \begin{proposition}\label{weaklifting}
   Let $f : R \to S$ be a ring homomorphism of Gorenstein rings with the weak MCM extending property. Then, 
   \begin{align*}
       f(\ca(R)) \subseteq \ca(S).
   \end{align*}
  \end{proposition}
  \begin{proof}
      Let $r \in \ca(R)$ and $N \in \MCM(S)$. By the weak MCM extending property, there is an $M \in \MCM(R)$ and $N' \in \MCM(S) $ such that $N'$ is a direct summand of $M \otimes_R S$ and $N \in \langle N' \rangle$. We have a commutative diagram
      \begin{align*}
	\xymatrix{
		M \ar[dr]_{\alpha} \ar[rr]^{r} && M\\
		& F \ar[ur]_{\beta}
	}
	\end{align*}
	with a free $R$-module $F$. Tensoring this commutative diagram with $S$, we get 
	\begin{align*}
	\xymatrix{
	N' \ar@{^{(}->}[r] &	M \otimes_R S \ar[dr]_{\alpha \otimes_R 1} \ar[rr]^{f(r)} && M \otimes_R S \ar@{->>}[r] &  N'\\
	&	& F \otimes_R S \ar[ur]_{\beta \otimes_R 1}
	}
	\end{align*}
	Note that the composition of three maps on the first row is still multiplication by $f(r)$ and $F \otimes_R S$ is a free $S$-module. Thus, $f(r)$ stably annihilates $N'$. By Lemma \ref{triangulatedlemma}, we deduce $f(r)$ stably annihilates $N$. As $N$ was arbitrary, we get the result.
  \end{proof}
  
  We present three applications of this proposition.
   
   \subsection{Milnor-Jung Type Formula For Branched Covers.} Let $S = k[\![x_0, \ldots, x_n ]\!]$ and $f \in ( x_0, \ldots, x_n )$. Let $S^\sharp = S [\![ y ]\!]$. Then, $R^\sharp = S^\sharp/(f + y^m)$ is called the $m$-branched cover of $R = S/(f)$.
   
   Note that $R \cong R^\sharp / (y)$ so that every $R$-module $M$ has an $R^\sharp$-module structure. Let $\pi: R^\sharp \to R$ be the natural surjection. For an $R^\sharp$-module $N$, we denote by $\pi N$ the $R$-module $N/yN$. We denote by $\Omega^\sharp$ the syzygy operator on $R^\sharp$. We have
   \begin{itemize}
    \item if $M \in \MCM(R)$, then $\Omega^\sharp M \in \MCM(R^\sharp)$ and,
    \item if $N \in \MCM(R^\sharp)$, then $\pi N \in \MCM(R)$
   \end{itemize}
   by keeping track of the depth of each module. The main ingredient of this subsection is the following lemma due to Kn{\"o}rrer \cite[Chapter 8.3.]{LW}, \cite[Chapter 12]{Y}.
   \begin{lemma}{\cite{K}}\label{Knorrerlemma}
    Let $m = 2$ so that $R^\sharp$ is the double-branched cover of $R$. Let $M \in \MCM(R)$ and $N \in \MCM(R^\sharp)$. Also assume that the characteristic of  $k$ is not $2$. Then,
    \begin{enumerate}
     \item $\pi\Omega^\sharp M \cong M \oplus \Omega M$ as $R$-modules, and
     \item $\Omega^\sharp \pi N \cong N \oplus \Omega^\sharp N$ as $R^\sharp$-modules.
    \end{enumerate}
   \end{lemma}
   
   \begin{remark}
    For $m > 2$, the second isomorphism is no longer true. See \cite{LD} for when exactly there is an isomorphism. However, the first isomorphism holds true by following the steps in the proof of Kn{\"orrer}'s lemma. In other words, $\pi$ has the weak MCM extending property.  
   \end{remark}
   
   \begin{theorem}\label{pullbacktheorem}
       Let $R$ and $R^\sharp$ be as above. Then, 
       \begin{align*}
           \pi (\ca(R^\sharp)) \subseteq \ca(R).
       \end{align*}
       When $m = 2$, the equality holds.
   \end{theorem}
   \begin{proof}
       The first part follows from Proposition \ref{weaklifting}. Let $m = 2$. Let $\pi(a) \in R  = \pi(R^\sharp)$. We show that if $\pi(a) \in \ca(R)$, then $a \in \ca(R^\sharp)$. Since $\pi(a)\Ext_R^1(\pi M, \pi N) = 0$ for all $m,n \in \MCM(R^\sharp)$, we have $a \Ext_{R^\sharp/y R^\sharp}^1(M/yM,N/yN)=0$. There are isomorphisms
       \begin{align*}
        \Ext_{R^\sharp/y R^\sharp}^1(M/yM,N/yN) \cong \Ext_{R^\sharp}^2(M/yM,N) \cong \Ext_{R^\sharp}^1(\Omega^\sharp(M/yM),N).
       \end{align*}
       Since $\Omega^\sharp(M/yM) \cong \Omega^\sharp \pi M \cong M \oplus \Omega^\sharp M$ by Lemma \ref{Knorrerlemma}(2), we get $a \Ext_{R^\sharp}^1(M,N) = 0$. Hence, $a \in \ca(R^\sharp)$.

   \end{proof}

   \begin{example}
    Let $R^\sharp = k[\![x,y,z]\!]/(x^3 + y^4 + z^2)$, $R_1 = k[\![x,y]\!]/(x^3 + y^4) $ and $R_2 = k[\![x,z]\!]/(x^3 +  z^2)$. Then, $R^\sharp$ is a double branched cover of $R_1$ and $4$-branched cover of $R_2$. We compute $\ca(R_1) = (x^2,xy,y^2)$ and $\ca(R_2) = (x,z)$. Hence, we have $\ca(R^\sharp) = (x^2, xy,y^2, z)$ by Theorem \ref{pullbacktheorem}. This example also shows that the theorem fails for $m > 2$. Indeed, $\ca(R^\sharp) \neq \pi_2^{-1} \ca(R_2) = (x,y,z)$ where $\pi_2: R^\sharp \to R_2$ is the projection map.
   \end{example}

   \begin{theorem}\label{mainbranchedtheorem}
    Suppose $k$ is an algebraically closed field with $\mathrm{char} k \neq 2$. Let $R = k[\![ x,y]\!] / (f)$ be the coordinate ring of a reduced curve singularity. For $S = k[\![x,y, z_1, \ldots, z_l]\!]/(f+z_1^2 + \ldots + z_l^2)$, we have
    \begin{align*}
     \dim_k \frac{S}{J_{S}} = 2 \dim_k \frac{S}{\ca(S)} -  r + 1
    \end{align*}
    where $r$ is the number of branches of the curve $f$ at its singular point.
   \end{theorem}
   
   \begin{proof}
 When $l = 0$, the assertion follows from the Milnor-Jung formula. So, let $l \geq 1$. Then, we have $z_1,\ldots, z_l \in J_S \subseteq \ca(S)$. Hence, there are isomorphisms
 \begin{align*}
  S/J_S \cong R/J_R, \quad S/\ca(S) \cong R/\ca(R)
 \end{align*}
 as $k$-vector spaces, the latter of which is shown by Theorem \ref{pullbacktheorem}. The theorem in the case $l = 0$ shows the assertion.
\end{proof}

   \subsection{Completion of a One Dimensional Gorenstein Ring} We show in Theorem \ref{maintheorem} that over a one dimensional reduced complete Gorenstein local ring, the cohomology annihilator coincides with the conductor ideal. In this subsection, we show that we can drop the completeness assumption. See also \cite{BHST}. Our result follows from their Theorem 4.5 but we give a different proof.
   
   \begin{lemma}
    Let $R$ be a local ring and $M \in \module(R)$. If $M$ is locally free at the minimal primes of $R$, then $M \oplus \Omega M$ has a well defined rank. In particular, if $R$ is reduced, then for any $M \in \MCM(R)$, $M \oplus \Omega M$ has a well defined rank.
   \end{lemma}
   \begin{proof}
       Let $p$ be a minimal prime. Consider the short exact sequence
       \begin{align*}
           0 \to \Omega M \to F \to M \to 0
       \end{align*}
       with $F$ a free $R$-module. Localizing at $p$, we get
       \begin{align*}
           0 \to (\Omega M)_p \to F_p \to M_p \to 0.
       \end{align*}
       We know that $M_p$ is free, hence this sequence splits and we get
       \begin{align*}
           F_p \cong (M \oplus \Omega M)_p
       \end{align*}
       which by definition tells us that $M \oplus \Omega M$ has well defined rank. 
   \end{proof}

   \begin{remark}
    A Noetherian local ring is Gorenstein if and only if its completion is Gorenstein. And note that the depth of a module is equal to the depth of its completion. In particular, $M \in \MCM(R)$ if and only if $\hat{M} \in \MCM(\hat{R})$. In other words, the combination of the above lemmas tells us that the inclusion $R \hookrightarrow \hat{R}$ has the weak MCM extending property if $\hat{R}$ is reduced. Indeed, we can see this as follows. If $N \in \MCM(\hat{R})$, then  $N \oplus \Omega_{\hat{R}} N$ has a well defined rank. Note that as $\hat{R}$ is reduced, $N$ is locally free at the minimal primes. Hence, $N \oplus \Omega_{\hat{R}} N \cong M \otimes_R \hat{R}$ for some $M \in \MCM(R)$ by \cite[Corollary 4.5]{KW}. 
   \end{remark}

   \begin{theorem}\label{completionintersection}
       Let $R$ be a one dimensional Gorenstein local ring whose completion $\hat{R}$ is reduced. Then,
       \begin{align*}
           \ca(R) = \ca(\hat{R}) \cap R.
       \end{align*}
   \end{theorem}

   \begin{proof}
       We know that the inclusion of $R$ into $\hat{R}$ has the weak MCM extending property. Hence, $\ca(R) \subseteq \ca(\hat{R}) \cap R$ by Proposition \ref{weaklifting}. For the other direction, we note the inclusion 
       \begin{align*}
           \sEnd_R(M) \hookrightarrow \widehat{\sEnd_R(M)} =  \underline{\mathrm{End}}_{\hat{R}}(\hat{M})  .
       \end{align*}
       If $r\in  \ca(\hat{R}) \cap R$, then it annihilates $\underline{\mathrm{End}}_{\hat{R}}(\hat{M})$ and thus $\sEnd_R(M)$ for all $M \in \MCM(R)$.
   \end{proof}
   
   \begin{theorem}
       Let $R$ be a one dimensional Gorenstein local ring with reduced completion. Then, $\ca(R) = \co(R)$.
   \end{theorem}
   
   \begin{proof}
     By Theorem \ref{maintheorem}, we know that $\ca(\hat{R}) = \co(\hat{R})$. Intersecting with $R$ we get the desired result by Theorem \ref{completionintersection} and \cite[A.1.Lemma]{WW}.
   \end{proof}
   
   \begin{remark}
    Let $R$ be a Gorenstein local ring and $R^h$ be its Henselization. For any maximal Cohen-Macaulay $R^h$-module $N$, there is an $M \in \MCM(R)$ such that $N$ is a direct summand of $M \otimes_R R^h$. This follows from Proposition 10.5 and Proposition 10.7 of \cite{LW}. In other words, the inclusion $R \hookrightarrow R^h$ has the weak MCM extending property. As an application of Proposition \ref{weaklifting}, we have the following theorem.
   \end{remark}
    
    \begin{theorem}
        For a Gorenstein local ring $R$, we have
        \begin{align*}
            \ca(R) \subseteq \ca(R^h).
        \end{align*}
    \end{theorem}
\section{Relation with Noncommutative Resolutions}
 \paragraph*{} Let $R$ be a Gorenstein ring. For this section, we say that a finite $R$-algebra $\Lambda$ is a noncommutative resolution of $R$ if $\Lambda \in \MCM(R)$ and $\gldim \Lambda < \infty$. Theorem \ref{maintheorem} tells us that the cohomology annihilator of a one dimensional complete reduced Gorenstein local ring is equal to the stable annihilator of its normalization. The normalization, in this case, is a one dimensional regular ring. In particular, it is a Cohen-Macaulay ring and since it is finite over the original ring, it is a maximal Cohen-Macaulay module. Hence, it is a noncommutative resolution of the given curve singularity. We have a noncommutative resolution whose stable annihilator over the singularity is equal to the cohomology annihilator of that singularity. In other words, in order to annihilate the singularity category, it is enough to stably annihilate a noncommutative resolution.
 
     \begin{lemma}\label{endoring1}
     Let $R$ be a Gorenstein ring and $M,N \in \MCM(R)$. If $N$ has a free summand and $\Lambda = \End_R(M \oplus N) \in \MCM(R)$, then 
     \begin{align*}
         \sann_R \Lambda \subseteq \sann_R M.
     \end{align*}
     \end{lemma}
     
     \begin{proof}
         We know that $\Hom_R(M \oplus N, M \oplus N)$ contains $\Hom_R(N,M)$ as a direct summand, which contains $\Hom_R(R,M) = M$ as a direct summand. The result follows from Lemma \ref{triangulatedlemma}.
     \end{proof}
 
 \paragraph*{}We, then, raise the following question: Is it always possible to find a nice algebra, preferably related to the singularities of our ring, whose stable annihilator is equal to the cohomology annihilator? We present examples where there is equality and we also give proofs that there is equality up to radicals.
 
 \begin{example}
     Let $R$ be a two dimensional invariant ring defined in Example \ref{twodimensionalexample}. Let $S * G$ be the skew group algebra generated as an $R$-module by elements of the form $(s,g)$ with $s \in S$ and $g \in G$ where the multiplication is defined by the rule
     \begin{align*}
         (s_1, g_1)(s_2, g_2) = (s_1g_1(s_2), g_1g_2).
     \end{align*}
     It is well-known that $R$ has only finitely many nonisomorphic indecomposable maximal Cohen-Macaulay modules. And $S * G \cong \End_R(S)$ is a noncommutative resolution of $R$ \cite{Aus}. We have 
     \begin{align*}
         \sann_R(S * G) = \sann_R( S^{|G|} ) = \sann_R(S) = \ca(R).
     \end{align*}
     So in this case, stable annihilator of a noncommutative resolution is equal to the cohomology annihilator.
     \end{example}
     
     \begin{example} Let $R$ be as in Example \ref{determinantalex}. The matrix $X$ defines a $k[X]$-linear map on the free $k[X]$-module $G$ of rank $n$. We also have every exterior powers $\bigwedge^j X : \bigwedge^j G \to \bigwedge^j G $ for $j = 1, \ldots, n$. Let $M_j = \coker \bigwedge^j X$. Then, each $M_j$ is annihilated by $\det X$ and hence has an $R$-module structure \cite{BuchsbaumRim}. Morover, $M_j$ is maximal Cohen-Macaulay for $j = 1, \ldots, n$. Note that $M_n$ is free of rank $1$.
     
     Let $\Lambda = \End_R(\bigoplus_{j = 1}^n M_j)$. Then, $\Lambda$ is a noncommutative resolution of $R$ \cite{BLV}. In particular, $\Lambda \in \MCM(R)$ and therefore $\ca(R) \subseteq \sann_R(\Lambda)$. On the other hand, as $\bigoplus_{j = 1}^n M_j$ has a free summand, Lemma \ref{endoring1} yields that
     \begin{align*}
         \sann_R \Lambda \subseteq \sann_R M = J_R \subseteq \ca(R) \subseteq \sann_R \Lambda .
     \end{align*}
     Hence, we have a noncommutative resolution $\Lambda$ whose stable annihilator is equal to the cohomology annihilator of $R$.
     \end{example}
     
     \begin{example}
         Let $R$ be as in Example \ref{torusinvariant}. Then, $\Lambda = \End_R(R \oplus S_1 \oplus S_2)$ is a noncommutative resolution of $R$. Hence, we have $\ca(R) \subseteq \sann_R \Lambda$. On the other hand, by Lemma \ref{endoring1}, we have $\sann_R \Lambda \subseteq \sann_R S_1 \cap \sann_R S_2 = \ca(R)$. Therefore, $\sann_R \Lambda = \ca(R)$ and we have a noncommutative resolution whose stable annihilator annihilates the entire singularity category.
     \end{example}

 \begin{proposition}\label{ncrtheorem}
  Let $R$ be a Gorenstein ring and $\Lambda$ be an $R$-algebra of finite global dimension $\delta$ such that $R \to \Lambda$ is a monomorphism. Suppose that $\Lambda$ is maximal Cohen-Macaulay as an $R$-module. Then, for any $M \in \module R$,
  \begin{align*}
      (\sann_R \Lambda)^{e + 1} \subseteq \sann_R M_{\Lambda}^{\st}
  \end{align*}
  where $M_{\Lambda} = M \otimes_R \Lambda$, $M_{\Lambda}^{\st}$ its image in $\sMCM(R)$ and $e$ its projective dimension as a $\Lambda$-module. If, moreover, the map $R \to \Lambda$ is a split monomorphism, then
  \begin{align*}
      (\sann_R \Lambda)^{\delta + 1} \subseteq \ca(R) \subseteq \sann_R \Lambda.
  \end{align*}
 \end{proposition}
 
 \begin{proof}
     If $M_\Lambda$ has projective dimension zero, then it is a summand of a direct sum of copies of $\Lambda$. We have $\sann_R \Lambda \subseteq \sann_R M_\Lambda$. Now, suppose that $M_\Lambda$ has projective dimension $e \geq 1$ and consider a short exact sequence
     \begin{align*}
         0 \to X \to G \to M_\Lambda \to 0
     \end{align*}
      so that $X$ has projective dimension $e - 1$. This short exact sequence produces a distinguished triangle
 \begin{align*}
    X^{\st} \to G \to (M_\Lambda)^{\st} \to \Omega^{-1}X^{\st}
 \end{align*}
 in $\sMCM(R)$. This gives us a distinguished triangle in $\sMCM(R)$ and using Lemma \ref{triangulatedlemma} we get the desired result by induction. If the canonical map $R \to \Lambda$ splits, then $M$ is a summand of $M_{\Lambda}$. Hence,
 \begin{align*}
      (\sann_R \Lambda)^{\delta + 1} \subseteq  (\sann_R \Lambda)^{e + 1}  \subseteq \sann_R M_{\Lambda} \subseteq \sann_R M.
 \end{align*}
 \end{proof}
 
 \begin{proposition}\label{endoring2}
     Let $R$ be a Gorenstein ring and $M, N \in \MCM(R)$. Suppose that $N$ has a free summand and the endomorphism ring $\Lambda = \End_R(M \oplus N)$ is also maximal Cohen-Macaulay. If $M$ is a strong generator of $\sMCM(R)$ with $\langle M \rangle_k = \sMCM(R)$, then 
     \begin{align*}
         (\sann_R \Lambda)^k \subseteq \ca(R) \subseteq \sann_R \Lambda.
     \end{align*}
 \end{proposition}
 
 \begin{proof}
     As $\Lambda \in \MCM(R)$, we know $ \ca(R) \subseteq \sann_R \Lambda$. By Lemma \ref{endoring1}, we know that $\sann_R \Lambda \subseteq \sann_R M$. So, by Lemma \ref{triangulatedlemma} we have $(\sann_R \Lambda)^k \subseteq (\sann_R M)^k \subseteq \ca(R)$.
 \end{proof}
 \begin{lemma}
  Let $R$ be a Gorenstein ring and $M \in \MCM(R)$ so that $\Lambda = \End_R(M)$ is also in $\MCM(R)$. Then, $(\sann_R M)^2 \subseteq \sann_R \Lambda$.
 \end{lemma}
 
 \begin{proof}
  Let $r, s \in \sann_R M$ so that there is a commutative diagram 
  \begin{align*}
      \xymatrix{
		M \ar[dr] \ar[rr]^{r} && M\\
		& F \ar[ur]
	}
  \end{align*}
  with $F$ being free over $R$ of rank $m$. Applying $\Hom_R(M,-)$, we get
  \begin{align*}
     \xymatrix{
		\Lambda \ar[dr] \ar[rr]^{r} && \Lambda\\
		& M^m \ar[ur]
	} 
  \end{align*}
  and extending this diagram we get
  \begin{align*}
     \xymatrix{
		\Lambda \ar[dr] \ar[rr]^{r} && \Lambda  \ar[rr]^{s} && \Lambda\\
		& M^m \ar[ur] \ar[dr] \ar[rr]^{s} && M^m \ar[ur] \\
		&& Q \ar[ur]
	} 
  \end{align*}
  where $Q$ is again a free $R$-module. This shows that $rs \in \sann_R \Lambda$.
 \end{proof}
 
 Recall that a maximal Cohen-Macaulay module $M$ over a Gorenstein ring $R$ is called \textit{$n$-cluster tilting} if
 \begin{align*}
     \mathrm{add}_RM &= \lbrace X \in \MCM(R) \; : \; \Ext_R^i(M, X) = 0 \rbrace = \lbrace Y \in \MCM(R) \; : \; \Ext_R^i(Y, M) = 0 \rbrace
 \end{align*}
 where $\mathrm{add}_RM$ is the subcategory of $\module R$ consisting of modules which are isomorphic to direct summands of finite direct sums of copies of $M$. If $R$ is an isolated Gorenstein singularity with $\dim R \geq 3$ and $M$ is a $(d-2)$-cluster tilting module, then $\Lambda = \End_R(M)$ is maximal Cohen-Macaulay and $\gldim \Lambda = \dim R$ \cite[Theorem 5.2.1.]{I}. Moreover, $M$ has a free summand. Combining the above discussions, we have
 \begin{theorem}
  If $R$ is an isolated Gorenstein singularity with $\dim R \geq 3$ and $M$ is a $(d-2)$-cluster tilting module, then $(\sann_R M)^{2\dim R} \subseteq \ca(R) \subseteq \sann_R M$. In particular, stable annihilator of a $(d-2)$-cluster tilting module defines the singular locus of $R$ if in addition we assume the hypotheses of \cite[Theorem 5.3.]{IT} or \cite[Theorem 5.4.]{IT}.
 \end{theorem}

\bibliography{esentepe}

\end{document}